\documentclass[11pt,a4paper,final]{amsart}

\usepackage[pdfauthor={D. Schrittesser, A. Tornquist},
            pdftitle={The Ramsey property and higher dimensional mad families},
            pdfproducer={Latex with hyperref}]{hyperref}

\usepackage{todonotes}

\usepackage[utf8]{inputenc}
\usepackage[T1]{fontenc}
\usepackage{lmodern}
\usepackage[english]{babel}
\usepackage{csquotes}

\usepackage{amsmath}
\usepackage{amstext}
\usepackage{amssymb}
\usepackage{amsthm}

\usepackage{mathrsfs}
\usepackage[all]{xy}
\usepackage{datetime}

\usepackage[notref,notcite]{showkeys}

\usepackage{datetime}

\usepackage[shortlabels]{enumitem}
\usepackage{accents}
\newlength{\dhatheight}

\makeatletter
\newcommand{\whathat}[1]{{%
  \mathpalette\double@widehat{#1}%
}}
\newcommand{\double@widehat}[2]{%
  \sbox\z@{$\m@th#1\widehat{#2}$}%
  \ht\z@=.9\ht\z@
  \widehat{\box\z@}%
}
\makeatother

\theoremstyle{plain}
\newtheorem{thm}{Theorem}[section]
\newtheorem{lem}[thm]{Lemma}
\newtheorem{prp}[thm]{Proposition}

\theoremstyle{definition}
\newtheorem{dfn}[thm]{Definition}

\newtheorem{fact}[thm]{Fact}

\theoremstyle{remark}

{\end{minipage}\end{equation}}

\newenvironment{eqpar*}{\begin{equation*}\begin{minipage}{0.8\columnwidth}}%
{\end{minipage}\end{equation*}}

\mathchardef\mhyphen="2D

\DeclareMathOperator{\lh}{\ell h}

\DeclareMathOperator{\ran}{ran}

\providecommand{\forces}{\Vdash}

\providecommand{\R}{R}

    \DeclareMathOperator{\s}{s}
    
    \DeclareMathOperator{\terminal}{tmnl}

    \newcommand{\restrict}{\!\upharpoonright\!}

    \newcommand{\baire}{\omega^\omega}

    \newcommand{\append}{{^\frown\!}}
    \newcommand{\incfcn}{[\omega]^\omega_{\bullet}}

\def\R{{\mathbb R}}

\def\bL{{\mathbb L}}

\author[Schrittesser]{David Schrittesser}
\address{David Schrittesser, Institute for Advanced Study in Mathematics, Harbin Institute of Technology, 92 West Da Zhi Street, Harbin City, Hei\-long\-jiang Province 150001, China, \emph{and}}

\address{Harbin Institute of Technology Suzhou Research Institute, Building K, 500 Nan Guandu Road, Suzhou City, Jiangsu Province 215104, China.}
\email{david@logic.univie.ac.at}

\author[Törnquist]{Asger Törnquist}

\address{Asger Törnquist, Department of Mathematical Sciences, University of Copenhagen, Universitetspark 5, 2100 Copenhagen, Denmark}
\email{asgert@math.ku.dk}

\title{The happy coexistence of mad families and Laver measurability}



\begin{document}

\begin{abstract}
Let $x$ denote a Laver real over $L$. We prove that in $L[x]$ there is a $\Pi^1_1$ infinite mad family. Since $\Pi^1_1$ and $\Sigma^1_2$ sets are Laver measurable in $L[x]$, this shows that there are examples of well-behaved classical pointclasses $\Gamma$, namely $\Gamma=\Pi^1_1$ and $\Gamma=\Sigma^1_2$, where $\Gamma$-uniformization and ``all sets in $\Gamma$ are Laver measurable'' hold, but there is a mad family in $\Gamma$. This result stands in contrast to that in \cite{pnas} it was obtained that for reasonable pointclasses, the $\Gamma$-Ramsey property together with uniformization implies that there are no mad families in $\Gamma$. 
\end{abstract}

\maketitle

\section{Introduction}

A.D.R. Mathias classic paper {\it ``Happy Families''} \cite{mathias} connects in a remarkable way (1) Ramsey-theoretic measurability properties for families of subsets of the natural numbers $\omega$, (2) the kind of forcing reals, now known as \emph{Mathias reals}, that arise from Mathias forcing, and (3) the combinatorial structures known as \emph{mad} families, that is, \emph{maximal almost disjoint families}, of subsets of $\omega$. The understanding of these connections have since been developed in many directions by various authors, of which \cite{brendle, chan-jackson, haga-schrittesser-toernquist, horowitz-shelah-no-mad,horowitz-todorcevic, neeman-norwood, asger} are perhaps the most immediately relevant to the present discussion.

\medskip

The present paper is a contribution to the understanding of the connection between (3) and the other two. Let us first discuss (3) and the forcing point of view (2); some remarks about the measurability view follows afterwards. (The precise definition the Mathias and Laver posets are reviewed in \ref{ss.forcingmeas} below.)

\medskip

One picture that emerges from Mathias' and subsequent work can be summarized as: Definable mad families do not co-exist happily with Mathias reals. An illustration of this, which we will consider prototypical, and which can be derived easily from the results of \cite{pnas}, is: If $x$ is a Mathias real over $L$ (G\"odel's constructible universe), then there are no $\Pi^1_1$ mad families $L[x]$ (and then there are no $\Sigma^1_2$ mad families in $L[x]$ either, by the trick in \cite{tornquist-pi}).

\medskip

By contrast, it turns out that definable mad families coexist just fine with some other forcing reals, including, but certainly not limited to, Cohen and Random reals; this is delightfully demonstrated in \cite{horowitz-shelah-no-mad}.

\medskip

Why this difference? The most obvious difference between the reals considered in \cite{horowitz-shelah-no-mad}, which \emph{can} coexist with definable mad families, and Mathias reals, is the fast growth of the latter. The proof in \cite{pnas} supports the suspicion that this is important, though other properties of the Mathias poset, especially the Prikry property (also called the ``pure decision property''), also seem to play a role.

\medskip

Among the classical forcing posets, the Laver poset is one whose generic reals grow very fast, and Laver forcing also has the Prikry property. So the natural question arises if Laver reals could replace Mathias reals in relation to the definability of mad families. That is, can Laver reals coexist with definable mad families, or are they similar to Mathias reals in that they do not allow such coexistence?

The purpose of this paper is to prove the following theorem which shows, somewhat surprisingly, that Laver reals on a very basic level can coexist with definable mad families:

\begin{thm}\label{t.lavermad}
Let $x$ be a Laver real over $L$. Then there is an infinite $\Pi^1_1$ mad family in $L[x]$.
\end{thm}

One may take this theorem as an expression of that it truly is the Ramsey-theoretic properties of Mathias forcing that make the proof in \cite{pnas}, and in turn \cite{HeLuoZhang}, work.

\subsection{Forcing and measurability properties}\label{ss.forcingmeas} There is a well-known connection between many classical forcing notions and measurability properties. The best-known are of course the connection between Cohen forcing and the Baire property, and Random forcing and Lebesgue measurability. The paper \cite{brendle-lowe} expands this greatly by generally considering measurability notions associated with so-called \emph{arboreal} forcing notions, a general class of forcing notions to which all considered here belong.

\medskip

We keep here a narrow focus and consider only the measurability notions associated with Mathias and Laver forcing, and refer to \cite{brendle-lowe} for the general theory. For the convenience of the reader, and for the purpose of fixing notation that will later be used in relation to the Laver poset, let us set down the otherwise well-known definitions of the Mathias and Laver posets:

\begin{dfn}\

\begin{enumerate}
\item The Mathias poset (originally defined in \cite{mathias-thesis,mathias}, but see also \cite[p.~524]{jech}), denoted $\R$, consists of conditions of the form $(a,A)$ where $a,A\subseteq\omega$, $a$ is finite and $A$ is infinite, and $a\sqsubseteq A$ (meaning that $a$ is an initial segment of $A$, that is, $\max(a)<\min (A\setminus a)$). The extension relation is $(b,B)\leq (a,A)$ iff $a\sqsubseteq b$ and $B\subseteq A$.

\item The Laver poset (originally defined in \cite{Laver1976}, but see also \cite[p.~565]{jech}), denoted $\bL$, consists of all \emph{stemmed Laver trees}, that is, infinite subtrees $p\subseteq\omega^{<\omega}$ where there is unique $s\in p$ (the ``stem'' of $p$) such that
$$
\hspace{1cm}(\forall t\in p)\ t\subseteq s\vee \big(s\subseteq t\wedge t\append i\in p\text{ for infinitely many } i\in\omega\big). 
$$
We will write $\s(p)$ for the stem of $p$. The ordering of $\bL$ is inclusion, i.e. $p\leq q$ just in case $p\subseteq q$ (so the stem can grow longer, and above the stem we may thin out the tree). If $s(p)=\emptyset$ (that is, $p$ is ``stemless'') is called simply a Laver tree.

When $G\subseteq\bL$ is a sufficiently generic filter (meaning that $G$ meets sufficiently many dense sets), then
$$
x_G\overset{def}=\bigcup_{p\in G} s(p)
$$
defines an element of $\omega^\omega$. When $G$ is $M$-generic for some transitive model $M$, we say that $x_G$ is a \emph{Laver real over $M$}.


\end{enumerate}
\end{dfn}

The associated measurability notions are defined as follows:

\begin{dfn}
\begin{enumerate}
\item For each $(a,A)\in\R$, define
$$
[a,A]=\{B\subseteq\omega\mid a\sqsubset B\subseteq A\}.
$$
Then a set $X\subseteq [\omega]^\omega$ is $\R$-measurable iff it is complete Ramsey, that is, if for all $(a,A)\in\R$ there is $B\subseteq A$ such that
$$
[a,B]\subseteq X\text{ or } [a,B]\cap X=\emptyset.
$$
Ellentuck's theorem (\cite{ellentuck}, but see also \cite[Theorem 19.?]{kechris}) states that being completely Ramsey is equivalent to having the Baire property in the \emph{Ellentuck topology}, whose basic open sets are $[a,A]$. 
\item For each $p\in\bL$, let $[p]$ denote the set of infinite branches through $p$. We will say that $X\subseteq \omega^\omega$ is \emph{Laver measurable} if for all $p\in\bL$ there is $q\leq p$ such that
$$
[q]\subseteq X\text{ or } [q]\cap X=\emptyset.
$$ 
\end{enumerate}
\end{dfn}

A. Miller has shown \cite{miller2012} that when $x$ is a Laver real over $L$ then all $\Sigma^1_2$ subsets of $\omega^\omega$ are Laver measurable (even in a very strong way). With this in hand we can rephrase Theorem \ref{t.lavermad} in terms of measurability as follows:

\begin{thm}\label{t.lavermad2}
When $x$ is a Laver real, then $L[x]$ is a model of ZFC in which all $\Sigma^1_2$ sets are Laver measurable, but there is a $\Pi^1_1$ mad family.
\end{thm}

One may read this theorem as saying that for pointclasses $\Gamma$ (in the sense of Moschovakis \cite{moschovakis}), $\Gamma$-Laver measurability and $\Gamma$-uniformization (the latter of which the pointclasses $\Pi^1_1$ and $\Sigma^1_2$ have outright), are not enough to ensure that there are no mad families in $\Gamma$. Contrast this with that, by \cite{pnas} and \cite{higher-dim}, for reasonable pointclasses it holds that the $\Gamma$-Ramsey property in conjunction with $\Gamma$-uniformization \emph{does} imply that there are no mad families in $\Gamma$.

\subsection*{Acknowledgments} David Schrittesser thanks the 
FWF for support through Ben Miller's project P29999 and Vera Fischer's START Prize Y1012, as well as the Government of Canada’s New Frontiers in Research Fund (NFRF), for support through grant NFRFE-2018-02164. Asger Törnquist gratefully acknowledges support for this research received from the Independent Research Fund Denmark through the grants DFF2 ``Operator algebras, groups, and quantum spaces'' and DFF1 grant no. 4283-00353B.

We thank Natasha Dobrinen, Stevo Todorčević, Michael Hrušák, Juris Steprāns, the participants of the Delta Workshop in Wuhan, 2023, as well as Jialiang He, Jintao Luo, and Shuguo Zhang, for interesting conversations on topics related to this paper.

\section{Proof of Theorems \ref{t.lavermad} and \ref{t.lavermad2} from the ``Main Lemma''}\label{laver}

In this section we first introduce some more background on Laver forcing, and we introduce a theorem due to A. Miller we'll need. Then we state without proof the ``Main Lemma'', which is the main technical step of the entire proof, and then see that Theorem \ref{t.lavermad} follows from the Main Lemma. (Theorem \ref{t.lavermad2} follows directly from Theorem \ref{t.lavermad}.)

\subsection{Notation and further facts about Laver forcing}

(1) Recall that the extension relation in $\bL$ is inclusion, that is $p\leq q$ if $p\subseteq q$. We will write $q\leq^* p$ just in case $q\leq p$ and $s(q)=s(p)$ (that is, the stem is unchanged).

(2) Given $p\in\bL$ and $t\in p$, we let $p/t$ denote the tree consisting of all $s\in p$ that are compatible with $t$. Note that $p/t$ is a stemmed Laver tree, whose stem extends $t$.

(3) Laver forcing has the \emph{Prikry property} (sometimes called ``pure decision'', see \cite[28.19]{jech}): If $p\forces\varphi_1\vee\cdots\vee\varphi_k$, then there is $q\leq^* p$ and $1\leq i\leq k$ such that $q\forces\varphi_i$.

(4) Laver forcing has ``continuous reading of names'' for reals in the sense of the following proposition:
\begin{prp}\label{p.contread}
 If $\tau$ is an $\bL$-name and $p\forces \tau\in\omega^\omega$, then there is $q\leq^* p$ and a continuous function $f:[q]\to\omega^\omega$, such that $q\forces f(x_G)=\tau$. Moreover, we can arrange that whenever $q'\forces f(x_G)_i=j$ for some $i,j\in\omega$ and $q'\leq q$, then actually $f(x)_i=j$ for all $x\in [q]$ with $x\supseteq s(q')$.
 \end{prp}

\begin{proof}[Sketch of proof]
This is essentially a standard fusion argument, but since the idea comes back again and again later, it is worth sketching how it goes.

\medskip

Find a maximal antichain $\mathcal A_0\subseteq p$ in the tree $p$ such that for every $t\in\mathcal A_0$ there is $j_t$ and $q_t\leq^* p/t$ such that $q_t\forces \tau_0=j_t$. Then for each $t\in\mathcal A_0$, find a maximal antichain $\mathcal A_{1,t}$ in $q_t$ such that for all $t'\in\mathcal A_{1,t}$ there is $j_{t'}$ and $q_{t'}\leq^* q_t/t'$ such that $q_{t'}\forces \tau_1=j_{t'}$, and let $\mathcal A_1=\bigcup_{t\in\mathcal A_0} A_{1,t}$. Continue to ``freeze'' the values of $\tau_0,\tau_1,\ldots$ in this way (using maximal antichains $\mathcal A_{i+1,t}\subseteq q_t$, where $t\in\mathcal A_i$). Then let
$$
q=\bigcap_{i}\bigcup_{t\in\mathcal A_i} q_t,
$$
and for $x\in [q]$ let $f(x)_i=j_t$ iff $x\supseteq t$ for some $t\in \mathcal A_i$.
\end{proof}

\subsection{A theorem of Arnold Miller}  For the proof of Theorem \ref{t.lavermad}, an important component comes from a theorem due to Miller, \cite{miller2012}, which we will now recall. In this context, a \emph{Hechler tree} is a tree $H\subseteq\omega^{<\omega}$ such that for every $t\in H$, the set
$$
\{i\in\omega: t\append i\in H\}
$$
is cofinite. (So a Hechler tree is a rather fat Laver tree). Given a Laver condition $p\in\bL$, we call $H\subseteq p$ a Hechler tree \emph{in $p$}, if for all $t\in H$ with $t\supseteq s(p)$ we have that
$$
\{i\in\omega: t\append i\in H\}
$$
is a cofinite subset of $\{i\in\omega: t\append i\in p\}$.

\begin{thm}[Miller, \cite{miller2012}]
(1) Let $A\subseteq \omega^\omega$ be an analytic set. Then either there is a Laver tree $p$ such that $[p]\subseteq A$, or there is a Hechler tree $H$ such that $[H]\cap A=\emptyset$.

(2) If $A$ is $\Sigma^1_1$ (lightface), then the following effective strengthening holds: Either there is a $\Delta^1_1$ Hechler tree $H$ such that $[H]\cap A=\emptyset$, or else there is a Laver tree $p$ such that $[p]\subseteq A$
\end{thm}

Note that while (2) above is not stated explicitly in \cite{miller2012}, it follows rather easily from Miller's proof, since the ordinal analysis in Miller's proof can be carried out within $L_{\omega_1^{CK}}$. From Miller's theorem we obtain the following analogue of \cite[Fact 3.3]{mrl}:

\begin{fact}\label{f.miller}
(1) If $A\subseteq\omega^\omega$ is an analytic set and 
$$
p'\forces_{\bL} x_G\in A,
$$
then there is $p\leq p'$ (indeed, $p\leq^* p'$) such that $[p]\subseteq A$.

(2) If $\psi(x,y)$ is a $\Pi^1_1$ formula, then the set
$$
\{(p,a)\in\bL\times\omega^\omega: p\forces_{\bL}\psi( x_G,\check a)\}
$$
is $\Pi^1_1$.
\end{fact}
\begin{proof}[Proof of Fact \ref{f.miller}]
(1) If there is no such $p$, Miller's theorem gives us a Hechler tree $H\subseteq p'$ in $p'$ such that $[H]\cap A=\emptyset$. Since the statement $\forall x\in [H]\  x\notin A$ is $\Pi^1_2$, and therefore absolute by Shoenfield's theorem, we must have that $H\forces x_G\notin A$, which contradicts that $H\leq p'$ and $p\forces x_G\in A$.

(2) Fix $a\in\omega^\omega$, and consider the set
$$
A_{a,p}=\{x\in [p]: \neg\psi(x,a)\},
$$
which is a $\Sigma^1_1(a)$ set. Then it follows exactly as in (1), but now using the effective version of Miller's theorem, that $p\forces \psi(\dot x_G,\check a)$ if and only if there exists $H\in\Delta^1_1(a,p)$ which is Hechler in $p$ and $[H]\cap A_{a,p}=\emptyset$.

Note that we now have:
$$
p\forces \psi(x_G,\check a)\iff (\exists H\in\Delta^1_1(p,a)\text{ Hechler in }p)(\forall x\in H)
\ \psi(x,a).
$$
By the Spector-Gandy theorem (see, e.g., \cite{mansfield}) the right hand side gives the desired $\Pi^1_1$ definition.
\end{proof}

\subsection{Main Lemma and proof of Theorem \ref{t.lavermad}} We now turn to the proof of Theorem \ref{t.lavermad}. The proof uses an idea similar to that of \cite{mrl}, but Laver forcing does not fit the mould of the general theorems there, so we have to move much more carefully. In particular, the proof only applies to almost disjointness, and not to graph-theoretic discreteness in general, as was the case in \cite{mrl}.

\medskip

The main technical step of the proof is establishing the following lemma:

\begin{lem}[The Main lemma]\label{l.main}
Let $p\in\bL$ and suppose $f:[p]\to[\omega]^\omega$ is continuous. Then there $q\leq p$ and a continuous $\tilde f:[q]\to [\omega]^\omega$ such that $\ran(\tilde f)$ is almost disjoint and $\tilde f(x)\subseteq f(x)$ for all $x\in [q]$.
\end{lem}
We postpone the proof of the Main Lemma to the next section, and proceed to use the lemma to prove Theorem \ref{t.lavermad}.

\begin{proof}[Proof of Theorem \ref{t.lavermad}, given the Main Lemma]
We will define a $\Sigma^1_2$ predicate $\varphi(x)$ which defines a mad family in $L[r]$ (and in $L$), whenever $r$ is a Laver real over $L$. By \cite{tornquist-pi}, it then follows that there is a $\Pi^1_1$ mad family in $L[r]$. We omit the painstaking verification that the predicate $\varphi$ is $\Sigma^1_2$, as this follows exactly as in \cite{mrl}, but using Fact \ref{f.miller} above in place of \cite[Fact 3.3]{mrl}.

We work in $L$ to define $\varphi$; and $<_L$ is used for the canonical well-ordering of $L$. Let $(p_\xi, f_\xi)$, $\xi<\omega_1$, be a $\Sigma^1_2$ enumeration of all pairs of $p\in\bL$ and $f:[p]\to [\omega]^\omega$ continuous. We will define by recursion on $\xi<\omega_1$ a sequence of almost disjoint families $\mathcal A_\xi$, all of which will be $\mathbf{\Sigma}^1_1$.

To begin the recursion, let $\mathcal A_0$ be any $\Pi^0_1$ infinite a.d. family. Suppose $\mathcal A_\gamma$ has been defined for all $\gamma<\xi$, for some given $\xi<\omega_1$. If
$$
p_\xi\forces (\forall\gamma<\xi)(\forall y\in \mathcal A_\gamma) |f(x_G)\cap y|<\infty
$$
then by Fact \ref{f.miller} there is $p\leq^* p_\xi$ such that
$$
(\forall x\in [p])(\forall\gamma<\xi)(\forall y\in \mathcal A_\gamma) |f(x)\cap y|<\infty.
$$
Let $(q,\tilde f)$ be the $<_L$-least pair satisfying the Main Lemma for some such $p$ and $f=f\restrict [p]$, and let $\mathcal A_\xi=\ran(\tilde f)\cup\bigcup_{\gamma<\xi}\mathcal A_{\xi}$. If no such $p$ exists, then let $\mathcal A_\xi=\bigcup_{\gamma<\xi} \mathcal A_\gamma$. It is straight-forward to verify that
$$
\mathcal A=\bigcup_{\xi<\omega_1} \mathcal A_\xi
$$
is an almost disjoint family, and by an argument similar to that of \cite{mrl}, there is a natural $\Sigma^1_2$ predicate $\varphi(x)$ defining $\mathcal A$.

We claim that $\varphi$ defines a mad family in $L[r]$ whenever $r$ is a Laver real. It is clear that $\varphi$ defines an almost disjoint family in $L[r]$. To see that $\varphi$ defines a \emph{maximal} almost disjoint family in $L[r]$, suppose, seeking a contradiction, that this is not the case. Then there is some $p\in\bL$ and some name $\sigma$ for an element $[\omega]^\omega$ such that
\begin{equation}\label{eq.forces}
p\forces (\forall x) (\varphi(x)\to |\sigma\cap x|<\infty).
\end{equation}
Since Laver forcing has continuous reading of names (Proposition \ref{p.contread}) we can assume there is a continous function $f:[p]\to [\omega]^\omega$ (in $L$) such that
$$
p\forces f(x_G)=\sigma.
$$
Then there is $\xi<\omega_1$ such that $(p,f)=(p_\xi,f_\xi)$. By Miller's theorem, there is $p'\leq^* p$ we have
$$
(\forall x\in [p'])(\forall\gamma<\xi)(\forall y\in \mathcal A_\gamma)\ |f(x)\cap y|<\infty.
$$
The Main Lemma ensures that for any such $p'$ there is $(q,\tilde f)$ with $q\leq^* p'$ (and so also $q\leq^* p$) such that $\tilde f(x)\subseteq f(x)$ for all $x\in [q]$ and $\ran(\tilde f)$ is almost disjoint. Let $<_L$-least such $(q,\tilde f)$. Then by definition of $\mathcal A$ we must have that $\ran(\tilde f)\subseteq \mathcal A$. But then $q$, and therefore $p$, can't force that $|f(x_G)\cap y|=\infty$ for all $y\in\mathcal A$ since in the forcing extension we must have that $\tilde f(x_G)\subseteq f(x_G)$ and $\tilde f(x_G)\in\mathcal A$. So $p\not\forces (\forall y\in\mathcal A)\ |y\cap f(x_G)|<\infty$, which contradicts (\ref{eq.forces}).\end{proof}

\section{Proof of the Main Lemma}

We now turn to the technical work of proving \ref{l.main}, the Main Lemma.

\medskip

{\it Note on notation}: Elements of $[\omega]^\omega$ are infinite subsets of $\omega$. Obviously, each $A\in [\omega]^\omega$ corresponds to a strictly increasing function $\omega\to\omega: n\mapsto A_n$, where $A_n$ denotes the $n$'th element of $A$. Sometimes it is more useful to think of elements of $[\omega]^\omega$ as strictly increasing functions, and sometimes it is more useful to think of them as sets. To clearly indicate to the reader what the current viewpoint is, we'll write $\incfcn$ for the set of \emph{strictly increasing functions} $\omega\to\omega$. The $_\bullet$ is there to remind the reader that we're taking the function point of view.

\begin{dfn}\label{d.decides}
Let $p\in\bL$, $i\in\omega$, and let $f:\omega^\omega\to\omega^\omega$ be a continuous function. We say that $p$ \emph{decides $f(x_G)_i$} if $p\forces f(x_G)_i=j$ for some $j\in\omega$.
\end{dfn}

Note that if $f\in\incfcn$ and $p$ decides $f(x_G)_i$, then, since Laver forcing has the Prikry property, there is $q\leq^* p$ which decides all $f(x_G)_k$ for $k\leq i$. Moreover, using Proposition \ref{p.contread}, for any $p\in\bL$ we can find $q\leq^* p$ such that whenever $p/t\forces f(x_G)_i=j$, then then for any $x\in [q/t]$ we have $f(x)_i=j$.

\subsection{Strategy for proving the Main Lemma.} We now comment on the idea behind the proof of the Main Lemma, as the proof is rather technical. For this discussion, and for the later proof of the Main Lemma, we make the following definition:

\begin{dfn}
Let $p\in\bL$ and $f:[p]\to\incfcn$ be continuous. We will call $t\in p$ an \emph{$f$-deciding node} (in $p$) if for all $p'\leq^* p$ and all $i\in\omega$ there is $q\leq^* p'$ which decides $f(x_G)_i$. That is, $t\in p$ is $f$-deciding if for every $i\in\omega$ the set
$$
\{q\leq^* p: q\text{ decides } f(x_G)_i\}
$$
is $\leq^*$-dense below $p$.
\end{dfn}

Looking now at the Main Lemma and what it will take to prove it, given $f:[p]\to \incfcn$, our job is to find a suitable $q\leq^* p$ and define $\tilde f:[q]\to\incfcn$ such that $\ran(\tilde f(x))\subseteq \ran(f(x))$ and $\ran(\tilde f)$ is an almost disjoint family (when identified with a subsets of $[\omega]^\omega$).

At any given $t\in p$ with $t\supseteq s(p)$, we will ask if $t$ is $f$-deciding or not. If it is, then we can essentially obtain all the information we need about $f$ ``locally'', in the sense that, for every $i\in\omega$ and $q\leq^* p/t$, there is a condition $q_i\leq^* q$ that decides $f(x_G)_i$. But if $t$ is \emph{not} $f$-deciding, then there is some $i_0$ that requires us to look farther down the tree (below $t$) to find nodes that decide $f(x_G)_{i_0}$. In this case, we will prove in Lemma \ref{l.notdeciding} below that there is a wellfounded tree $T\subseteq p/t$ whose terminal nodes decide $f(x_G)_{i_0}$, and, crucially, we can arrange that \emph{different} terminal nodes of $T$ decide values for $f(x_G)_{i_0}$ that are \emph{different}.

Using what was just observed, we will then construct (in Lemma \ref{l.thetree} below) a sequence of well-founded subtrees $T_0\subsetneq T_1\subsetneq\ldots\subseteq p$, such that $T_\infty=\bigcup_{i\in\omega} T_i$ is in $\bL$ and $T_\infty\leq^* p$, and each $T_i$-terminal node of $t\in\terminal(T_i)$ decides $f(x_G)_{\eta_t}$ (for some $\eta_t\in\omega$) to be a unique value (that is, if $t'\neq t$ are terminal nodes in $T_i,T_j$ respectively, then the values decided for $f(x_G)_{\eta_{t'}}$ and $f(x_G)_{\eta_t}$ are different). As observed after Definition \ref{d.decides}, we can then find $q\leq T_\infty$ such that for $t\in\terminal(T_i)$ there is $\phi_t\in\omega$ so that for any $x\in[q/t]$ we must have $f(x)_{\eta_t}=\phi_t$. Then we define $\tilde f:[q]\to [\omega]^{\omega}$ to be
$$
\tilde f(x)=\{\phi_t: t\subseteq x\wedge (\exists i\in\omega)\ t\in\terminal(T_i)\}.
$$
This $\tilde f$ will have $\tilde f(x)\subseteq f(x)$, and since we have arranged that for $t\in\terminal(T_i)$ and $t'\in\terminal(T_j)$, if $t\neq t'$ then $\phi_t\neq\phi_{t'}$ we must have that $\tilde f(x)\cap\tilde f(x')$ is finite whenever $x\neq x'$.

\subsection{Proving the Main Lemma}

The outcome of carrying out the strategy outlined above, where one analyses what can be decided ``locally'' versus what needs a further exploration of the tree, can be neatly packaged into the following Lemma:

\begin{lem}\label{l.thetree}
Let $p\in\bL$ and let $f:[p]\to\incfcn$ be continuous. Then there is a sequence $T_0\subseteq T_1\subseteq\ldots\subseteq T_n\subseteq\ldots\subseteq p$ of well-founded subtrees of $p$, such that $T_\infty=\bigcup_{i\in\omega} T_i$ is a Laver condition with $T_\infty\leq^* p$, and such that for every $i\in\omega$ we have:
\begin{enumerate}
\item Every non-terminal node $t\in T_i$ with $t\supseteq s(p)$ has infinitely many immediate extensions in $T_i$.
\item Every terminal node in $T_i$ has infinitely many immediate extensions in $T_{i+1}$.
\item There are functions $\eta_i, \phi_i:\terminal(T_i)\to\omega$ and $\rho_i:\terminal(T_i)\to\bL$ such that for every $t\in\terminal(T_i)$ we have
\begin{enumerate}
\item $s(\rho_i(t))=t$;
\item $\rho_i(t)\forces f(x_G)_{\eta_i(t)}=\phi_i(t)$;
\item $T_{i+1}/t\subseteq\rho_i(t)$;
\item If $t'\in\terminal(T_j)$ for some $j\in\omega$ and $\phi_i(t)=\phi_j(t')$, then $i=j$ and $t=t'$.
\end{enumerate}
\end{enumerate}
\end{lem}

Let us once again postpone the technicalities, and prove the Main Lemma granted Lemma \ref{l.thetree}.

\begin{proof}[Proof of the Main Lemma]
Let $p\in\bL$ and let $f:[p]\to\incfcn$ be continuous. Let $(T_i)_{i\in\omega}$ and  $T_\infty$, as well as $\eta_i,\phi_i,\rho_i$, be as in Lemma \ref{l.thetree}. Let $q\leq^* T_\infty$ be such that when $t\in\terminal(T_i)$, then for all $x\in [q/t]$ we have $f(x)_{\eta_i(t)}=\phi_i(t)$. That such a $q$ exists by remark that follows Definition \ref{d.decides}.

Define $\tilde f:[q]\to [\omega]^\omega$ by 
$$
\tilde f(x)=\{\phi_i(t): i\in\omega\wedge t\subseteq x\wedge t\in\terminal(T_i)\}.
$$
Then $\tilde f(x)\subseteq \ran(f(x))$ for every $x\in [q]$ since $f(x)_{\eta_i(t)}=\phi_i(t)$ whenever $t\in\terminal(T_i)$. To see that $\ran(\tilde f)$ is almost disjoint, let $x,x'\in [q]$ and suppose $x\neq x'$. Let $m\in\omega$ be largest such $x\restrict m=x'\restrict m$. If $x\restrict m\subsetneq t\subseteq x$ and $x'\restrict m\subsetneq t'\subseteq x'$ and $t\in\terminal(T_i), t'\in\terminal(T_j)$, then part (d) of (3) in Lemma \ref{l.thetree} guarantees that $\phi_i(t)\neq\phi_j(t')$. It follows that
$$
\tilde f(x)\cap\tilde f(x')=\{\phi_k(t): k\in\omega\wedge t\subseteq x\restrict m\wedge t\in\terminal(T_k)\}
$$
which is finite since there are only finitely many $t\subseteq x\restrict m$. Thus $\ran(\tilde f)$ is an almost disjoint family.
\end{proof}

All that remains now is to prove Lemma \ref{l.thetree}. How the construction of $T_{i+1}$ from $T_i$ in Lemma \ref{l.thetree} will proceed below a given terminal $t\in T_i$ will depend on whether $t$ is $f$-deciding or not. The next lemma will tell us what to do in the case when it is not:

\begin{lem}\label{l.notdeciding}
Let $p\in\bL$, and let $f:[p]\to\baire$ be a continuous function. Suppose there is $i_0\in\omega$ so that no $q\leq^* p$ decides $f(x_G)_{i_0}$. Then there is 
\begin{enumerate}[label=$(\arabic*)$]
\item a wellfounded tree $T\subseteq p$;
\item an injection $\vartheta:\terminal(T)\to\omega$;
\item a function $\rho:\terminal(T)\to\bL$ with $\rho(t)\leq^* p/t$ for all $t\in\terminal(T)$;
\end{enumerate}
such that for every $t\in T$ with $t\supseteq s(p)$ we have:
\begin{enumerate}[label=$(\alph*)$]
\item if $t$ is terminal in $T$ then $\rho(t)\forces f(x_G)_{i_0}=\vartheta(t)$.
\item if $t$ is not terminal in $T$, then no $q\leq^* p/t$ decides $f(x_G)_{i_0}$.
\item if $t$ is not terminal in $T$, then $t$ has infinitely many immediate successors in $T$, and exactly one of the following hold: 

\begin{enumerate}[label=$(\roman*)$]
\item All immediate successors $t'\in T$ of $t$ are terminal; note that in this case $
\vartheta(\{ t'\in T : t'\supset t\})$ is infinite.

\item No immediate extension $t'\in T$ of $t$ is terminal.

\end{enumerate}
\end{enumerate}
\end{lem}

\begin{proof} To start the proof, we claim the following:

\medskip

{\bf Claim 1:} Under the assumptions of the Lemma: There is a well-founded tree $T'\subseteq p$ such that
\begin{enumerate}
\item if $t\in T'$ and $s(p)\subseteq t$, then no $q\leq^* p/t$ decides $f(x_G)_{i_0}$;
\item if $t\supseteq s(p)$ is not terminal in $T'$, then $t$ has infinitely many immediate extensions in $T'$;
\item if $t\in T'$ is terminal then the set
$$
A_t=\{i\in\omega: t\append i\in p\wedge (\exists q\leq^* p/{t\append i})\ q\text{ decides } f(x_G)_{i_0}\}
$$
is infinite.
\end{enumerate}

{\it Proof of Claim 1:} Let $T'$ be a maximal subtree of $p$ with $s(p)\in T$ such that (1) and (2) above hold. (That there is a subtree for which (1) and (2) hold is clear: $\{t:t\subseteq s(p)\}$ is such a tree.)  We just need to see that then (3) also holds for such a maximal $T'$. But if (3) fails for some terminal $t\in T'$, then there must be infinitely many $i$ with $t\append i\in p$ and no $q\leq^* p/t\append i$ which decides $f(x_G)_{i_0}$, and then clearly $T'$ is clearly not a \emph{maximal} subtree of $p$ for which (1) and (2) hold.

Finally, to see that $T'$ is wellfounded, suppose, seeking a contradiction, that $T'$ has an finite branch $x\in [T']$. Since $f$ is continuous, there must be some $n\in\omega$ such that $f(z)_{i_0}=f(x)_{i_0}$ for all $z\in [p/x\!\restrict\! n]$ (and we can of course assume that $n\geq\lh(s(p))$). Then $p/x\restrict n$ decides $f(x_G)_{i_0}$, contradicting $x\restrict n\in T'$ and $x\restrict n\supseteq s(p)$. \qed$_{\text{Claim 1}}$

\smallskip

{\bf Claim 2:} With $T'$ as in the previous claim we have: Given $t\in T'$ terminal and $k\in\omega$, there are only finitely many $i\in\omega$ such that some $q\leq^* p/t\append i$ decides that $f(x_G)_{i_0}=k$.

\smallskip
{\it Proof of Claim 2:} Otherwise we can construct $q\leq^* p/t$ which decides that $f(x_G)_{i_0}=k$, contradicting that $t\in T'$.\qed$_{\text{Claim 2}}$

\medskip

Using the Claim 2, we can easily build $T\supseteq T'$, as well as injections $\terminal(T)\to \omega:t\mapsto j_t$ and $\terminal(T)\to \bL:t\mapsto p_t$, such that each terminal node of $T$ is an immediate extensions of a terminal node in $T'$, and $p_t\leq^* p/t$ is a condition such that $p_t\forces f(x_G)_{i_0}=j_t$.
\end{proof}

Next we define the objects that we will need when we encounter a node in $p$ which \emph{is} $f$-deciding.

\begin{dfn} Let $p\in\bL$ and let $f:[p]\to\incfcn$ be continuous. A finite sequence $\vec u\in p^{<\omega}$ will be called $f$-deciding if $\vec u_j$ is $f$-deciding for all $j<\lh(\vec u)$ and $s(p)\subseteq \vec u_0\subsetneq\ldots\subsetneq \vec u_{\lh(\vec u)-1}$. We denote by $\Gamma=\Gamma_{p,f}$ the set of all $f$-deciding sequences.
\end{dfn}

The next lemma is really a recursive definition.

\begin{lem}\label{l.deciding}
With $p$ and $f$ as in the previous definition: We can associate to each $(\vec{u},\vec i)\in\Gamma\times\omega^{\lh(\vec u)}$ the following objects:
\begin{enumerate}[(a)]
\item a Laver condition $q_{\vec{u},\vec i}\leq^* p/\vec{u}_{\lh(\vec u)-1}$;
\item a strictly increasing function $g_{\vec u,\vec i}:\omega\to\omega$;
\end{enumerate}
such that the following holds: 
\begin{enumerate}[(*)]
\item Given $\vec u\append v\in\Gamma$ and $\vec i\in\omega^{\lh(\vec u)-1}$, the conditions $q_{\vec u\append v,\vec i\append j}$, where $j$ ranges over $\omega$, form a $\leq^*$-descending sequence of conditions below $q_{\vec u,\vec i}/s(v)$ such that $q_{\vec u\append v,\vec i\append j}\forces f(x_G)_{j}=g_{\vec u\append v,\vec i\append j}$.
\end{enumerate}
\end{lem}
\begin{proof}
The definition of these objects can be done by recursion on $\lh(\vec u)$.
\end{proof}

{\it Remark.} The objects defined in the previous Lemma are (in most situations) not uniquely determined by the requirements in the Lemma, so for our purposes below, \emph{we fix once and for all} $(\vec u,\vec i)\mapsto (q_{\vec{u},\vec i},g_{\vec u,\vec i})$ so that Lemma \ref{l.deciding} holds.

\medskip

The last preparatory step for the proof of Lemma \ref{l.thetree} is the following essentially obvious proposition:

\begin{prp}\label{p.refine}
Let $\mathcal A=\{A_i:i\in\omega\}$ be a countable family of infinite subsets of $\omega$. Then there is a countable family $\mathcal B=\{B_i:i\in\omega\}$ of infinite subsets of $\omega$ such that
\begin{enumerate}
\item $(\forall i\in\omega)\ B_i\subseteq A_i$;
\item $B_j\subseteq A_i$ implies that $i=j$;
\item the family $\mathcal B$ consists of pairwise disjoint sets.
\end{enumerate}
\end{prp}
We will call a family $\mathcal B$ as in the previous proposition a \emph{disjoint refinement} of the family $\mathcal A$.


\begin{proof}[Proof of Lemma \ref{l.thetree}]
For use in the proof, we first define some auxiliary objects:

\begin{enumerate}
\item For $t\in p$ with $s(p)\subseteq t$, let $\vec u(t)$ denote the maximal sequence with $s(p)\subseteq u(t)_0\subsetneq u(t)_1\subsetneq\ldots\subseteq u(t)_{\lh(u(t))-1}\subsetneq t$ which is $f$-deciding (and let $\vec u(t)=\emptyset$ if no $s(p)\subseteq u\subseteq t$ is $f$-deciding).

\item {\it When $t$ is $f$-deciding} then we associate to $t$ the objects $q_{\vec u(t),\vec i}$ and $g_{\vec u(t),\vec i}$ for each $\vec i\in \omega^{\lh(\vec u(t)}$.

\item {\it When $t$ is not $f$-deciding} then for each $\vec i\in \omega^{\lh\vec u(t)}$, we associate to $t$ the objects $T_{t,\vec i} \subseteq q_{\vec u(t),\vec i}/t$, and $\vartheta_{t,\vec i}:\terminal(T_{t,\vec i})\to\omega$, as well as $\rho_{t,\vec i}:\terminal(T_{t,\vec i})\to\bL$, that we obtain by applying Lemma \ref{l.notdeciding} with $q_{\vec u(t),\vec i}/t$ playing the role of $p$ in that Lemma.
\item For every $t\in p$ with $t\supseteq s(p)$ and $\vec i\in \omega^{\lh(t)}$, the sets $A_{t,\vec i}\subseteq \omega$ are defined by
$$
A_t=
\begin{cases}
\ran(g_{\vec u(t),\vec i}) & \text{ if } t\text{ is } f\text{-deciding};\\
\ran(\vartheta_{t,\vec i}) & \text{ if } t \text{ is not } f\text{-deciding}.
\end{cases}
$$
\item Let $\mathcal A=\{A_{t}: t\in p\wedge s(p)\subseteq t\}$, and let $\mathcal B=\{B_t:t\in p\wedge s(p)\subseteq t\}$ be a disjoint refinement of $\mathcal A$.
\item Letting $n=\lh(s(p))$, for each $t\in p$ with $t\supseteq s(p)$ we define 
$$
\vec i_t=\langle \min B_{t\,\restrict\,n},\min B_{t\,\restrict\,(n+1)},\ldots, \min B_{t}\rangle.
$$
\end{enumerate}

To prove Lemma \ref{l.thetree}, we define wellfounded trees $T_0\subsetneq T_1\subsetneq\ldots\subseteq p$ recursively as follows: Let $T_0=\{t: t\subseteq s(p)\}$. If $T_i$ has been defined (and is wellfounded), then define for each $t\in\terminal(T_i)$ a wellfounded subtree $T_{i,t}\supseteq p/t$ as follows:
$$
T_{i,t}=\begin{cases}
\{r\in p/t: r\subseteq t\vee (\exists k)\  r=t\append k\in q_{\vec u(t),\vec i_t}\} & \text{ if } t\text{ is } f\text{-deciding};\\
T_{\vec u(t),\vec i_t} &\text{ if } t\text{ is not } f\text{-deciding}.
\end{cases}
$$
Let $T_{i+1}=\bigcup\{ T_{i,t}: t\in\terminal(T_i)\}$. Then (1) of (2) of Lemma \ref{l.thetree} clearly hold. To define $\eta_i,\phi_i:\terminal(T_i)\to\omega$ and $\rho_i:\terminal(T_i)\to\bL$, let $\phi_i(t)=\min B_t$, let 
$$
\rho_i(t)=\begin{cases}
q_{\vec u(t),\vec i_t} & \text{ if } t \text{ is } f\text{-deciding};\\
\rho_{t,\vec i_t} & \text{ if } t \text{ is not } f\text{-deciding},
\end{cases}
$$
and let $\eta_i(t)\in\omega$ be such that $\rho_i(t)\forces f(x_G)_{\eta_i(t)}=\phi_i(t)$. Clearly (a) and (b) of (3) in Lemma \ref{l.thetree} hold. An easy induction on $i\in\omega$ shows that (c) also holds. Finally, for (d), note that $\phi_i(t)\in B_t$, so since $\mathcal B$ consists of pairwise disjoint sets, (d) holds.\end{proof}

\bibliography{schrittesser-toernquist-ramsey-property-higher-mad}{}

\providecommand{\bysame}{\leavevmode\hbox to3em{\hrulefill}\thinspace}
\providecommand{\MR}{\relax\ifhmode\unskip\space\fi MR }
\providecommand{\MRhref}[2]{%
  \href{http://www.ams.org/mathscinet-getitem?mr=#1}{#2}
}
\providecommand{\href}[2]{#2}
\begin{thebibliography}{10}

\bibitem{brendle}
J{\"o}rg Brendle, Osvaldo Guzm{\'a}n, Michael Hru{\v s}{\'a}k, and Dilip
  Raghavan, \emph{Combinatorial properties of {MAD} families},
  \href{https://arxiv.org/abs/2206.14936}{\nolinkurl{arXiv:2206.14936
  [math.LO]}}, 2022.

\bibitem{brendle-lowe}
J\"{o}rg Brendle and Benedikt L\"{o}we, \emph{Solovay-type characterizations
  for forcing-algebras}, J. Symbolic Logic \textbf{64} (1999), no.~3,
  1307--1323. \MR{1779764}

\bibitem{chan-jackson}
William Chan, Stephen Jackson, and Nam Trang, \emph{Almost disjoint families
  under determinacy}, Advances in Mathematics \textbf{437} (2024), 109410.

\bibitem{ellentuck}
Erik Ellentuck, \emph{A new proof that analytic sets are {R}amsey}, J. Symbolic
  Logic \textbf{39} (1974), 163--165. \MR{349393}

\bibitem{haga-schrittesser-toernquist}
Karen~Bakke Haga, David Schrittesser, and Asger T\"{o}rnquist, \emph{Maximal
  almost disjoint families, determinacy, and forcing}, Journal of Mathematical
  Logic \textbf{22} (2022), no.~01, 2150026.

\bibitem{HeLuoZhang}
Jialiang He, Jintao Luo, and Shuguo Zhang, \emph{Ramsey regularity implies no
  mad families without uniformization},
  \href{https://arxiv.org/abs/2510.17399}{\url{arXiv:2510.17399 [math.LO]}},
  2025.

\bibitem{horowitz-shelah-no-mad}
Haim Horowitz and Saharon Shelah, \emph{On the non-existence of mad families},
  Arch. Math. Logic \textbf{58} (2019), no.~3-4, 325--338. \MR{3928385}

\bibitem{horowitz-todorcevic}
Haim Horowitz and Stevo Todorcevic, \emph{Compact sets of {B}aire class one
  functions and maximal almost disjoint families}, 2019,
  \href{https://arxiv.org/abs/1905.05898}{\nolinkurl{ arXiv:1905.05898
  [math.LO]}}.

\bibitem{jech}
Thomas Jech, \emph{Set theory}, Springer Monographs in Mathematics,
  Springer-Verlag, Berlin, 2003, The third millennium edition, revised and
  expanded. \MR{1940513}

\bibitem{kechris}
Alexander~S. Kechris, \emph{Classical descriptive set theory}, Graduate Texts
  in Mathematics, vol. 156, Springer-Verlag, New York, 1995. \MR{1321597}

\bibitem{Laver1976}
Richard Laver, \emph{On the consistency of borel's conjecture}, Acta
  Mathematica \textbf{137} (1976), no.~1, 151--169.

\bibitem{mansfield}
Richard Mansfield and Galen Weitkamp, \emph{Recursive aspects of descriptive
  set theory}, Oxford Logic Guides, vol.~11, The Clarendon Press, Oxford
  University Press, New York, 1985. \MR{786122}

\bibitem{mathias-thesis}
A.R.D. Mathias, \emph{On a generalization of {Ramsey}'s {Theorem}}, Ph.D.
  thesis, Cambridge, 1970, Available at
  \url{https://www.dpmms.cam.ac.uk/~ardm/}.

\bibitem{mathias}
\bysame, \emph{Happy families}, Ann. Math. Logic \textbf{12} (1977), no.~1,
  59--111.
  \MR{{\href{http://www.ams.org/mathscinet-getitem?mr=491197}{0491197}}}

\bibitem{miller2012}
Arnold~W. Miller, \emph{Hechler and {Laver} trees},
  \href{https://arxiv.org/abs/1204.5198}{\url{arXiv:1204.5198 [math.LO]}},
  2012.

\bibitem{moschovakis}
Yiannis~N. Moschovakis, \emph{Descriptive set theory}, second ed., Mathematical
  Surveys and Monographs, vol. 155, American Mathematical Society, Providence,
  RI, 2009. \MR{2526093}

\bibitem{neeman-norwood}
Itay Neeman and Zach Norwood, \emph{Happy and mad families in
  {$L(\mathbb{R})$}}, J. Symb. Log. \textbf{83} (2018), no.~2, 572--597.
  \MR{3835078}

\bibitem{mrl}
David Schrittesser and Asger T{\"o}rnquist, \emph{Definable maximal discrete
  sets in forcing extensions}, Math. Res. Lett. \textbf{25} (2018), no.~5,
  1591--1612. \MR{3917741}

\bibitem{pnas}
\bysame, \emph{The {R}amsey property implies no mad families}, Proc. Nat. Acad.
  Sci. U.S.A. \textbf{116} (2019), no.~38, 18883--18887.

\bibitem{higher-dim}
David Schrittesser and Asger T{\"o}rnquist, \emph{The {R}amsey property and
  higher dimensional mad families},
  \href{https://arxiv.org/abs/2003.10944}{\url{arXiv:2003.10944 [math.LO]}},
  2024.

\bibitem{tornquist-pi}
Asger T{\"o}rnquist, \emph{{$\Sigma^1_2$} and {$\Pi^1_1$} mad families}, J.
  Symbolic Logic \textbf{78} (2013), no.~4, 1181--1182. \MR{3156517}

\bibitem{asger}
\bysame, \emph{Definability and almost disjoint families}, Adv. Math.
  \textbf{330} (2018), 61--73. \MR{3787540}

\end{thebibliography}
\bibliographystyle{amsplain}
 
\end{document}